\documentclass{amsart}
\usepackage{amsfonts}
\usepackage{amsmath,amscd}
\usepackage{amssymb}
\usepackage{amsthm}
\usepackage{newlfont}

\long\def\alert#1{\parindent2em\smallskip\hbox to\hsize%
{\hskip\parindent\vrule%
\vbox{\advance\hsize-2\parindent\hrule\smallskip\parindent.4\parindent%
\narrower\noindent#1\smallskip\hrule}\vrule\hfill}\smallskip\parindent0pt}
 \newtheorem{thm}{Theorem}[section]
 \newtheorem{cor}[thm]{Corollary}
 \newtheorem{lem}[thm]{Lemma}
 \newtheorem{prop}[thm]{Proposition}
 \theoremstyle{definition}
 
 \theoremstyle{remark}

 \numberwithin{equation}{section}

\begin{document}

\title[nilpotent Lie algebra with the derived subalgebra of dimension two ]
{ Certain functors of nilpotent Lie algebra with the derived subalgebra of dimension two}

\author[F. Johari]{Farangis Johari}
\author[P. Niroomand]{Peyman Niroomand}

\address{Department of Pure Mathematics\\
Ferdowsi University of Mashhad, Mashhad, Iran}
\email{farangis.johari@mail.um.ac.ir, farangisjohary@yahoo.com}

\address{School of Mathematics and Computer Science\\
Damghan University, Damghan, Iran}
\email{niroomand@du.ac.ir}



\date{\today}

\date{\today}


\begin{abstract}
By considering the nilpotent Lie algebra with the derived subalgebra of dimension $ 2$, we compute some functors include the Schur multiplier, the exterior square and the tensor square of these Lie algebras. We also give the corank of such Lie algebras.

 \end{abstract}

\maketitle

\section{Introduction and  Motivation and
Preliminaries}
Let $ F $ be an algebraic closed field and $ L $ be a nilpotent Lie algebra over $ F.$ Following \cite{ba1,2b, mon, ni, ni1,ni2, ni3}, the  Schur multiplier  of a Lie algebra $ L,$ $ \mathcal{M}(L),$ can be defined as $ \mathcal{M}(L)\cong R\cap F^2/[R,F]$ where $L\cong F/R $ for  a free Lie algebra $ F. $
There is a wide literature about Schur multiplier of $ L $ and the reader can find some more information on this topic for instance in \cite{ba1,2b, mon, ni, ni1,ni2, ni3}.

By looking \cite{elis1}, suppose $ L\otimes L $ is used to denote the tensor square of a Lie algebra $ L $. The exterior square $ L\wedge L $ is obtained from $ L\otimes L $ by imposing the additional relation $ l\otimes l=0 $ for all $ l\in L $ and the image of $ l\otimes l' $ is denoted by $ l\wedge l' $ for all $ l,l'\in L. $ From \cite{elis1}, let $ L\square L $ be the subalgebra of $ L\otimes L $ generated  by the elements $ l\otimes l $ for all $ l\in L. $ Clearly, $ L\wedge L\cong  L\otimes L /L\square L$ and by a famous result of Ellis (see Lemma \ref{j1}), the Schur multiplier of $ L $ is isomorphic to  the kernel of  the commutator map  $ \kappa' :l\wedge l_1 \in L\wedge L \longmapsto [l,l_1]\in L^2 $.

The goal of this paper is to give a complete description of the structure of some functors, such as  the Schur multiplier, the exterior square and the tensor square for a non-abelian nilpotent Lie algebra $ L $ with
the derived subalgebra of dimension at most $ 2. $

For a Lie algebra $L$ we denote
by $L^{(ab)}$ the factor Lie algebra $L/L^2,$ which is always abelian.
The following lemma gives the structure of Schur multiplier with respect to the direct sum of two Lie algebras.

\begin{lem}\label{1}
 \cite[Proposition 3]{ba1}
Let $ A$ and $ B $ be two Lie algebras. Then
\[ \mathcal{M}(A\oplus B)\cong   \mathcal{M}(A) \oplus  \mathcal{M}(B) \oplus (A^{(ab)}\otimes B^{(ab)}).\]
\end{lem}

We denote an abelian Lie algebra of dimension $n$ and Heisenberg Lie algebra by $A(n)$ and $H(m),$ respectively.
Schur multipliers of abelian and Heisenberg algebras are well known. (See for instance \cite[ Lemma 2.6]{ni3}.)
\begin{lem} \label{ss}
We have
\begin{itemize}
\item[$(i)$]
$\dim(\mathcal{M}(A(n))) = \dfrac{1}{2}
n(n -1),$
\item[$(ii)$]$\dim \mathcal{M}(H(1))=2,$
\item[$(iii)$]$\dim \mathcal{M}(H(m))=2m^2-m-1$ for all $ m\geq 2.$
\end{itemize}
\end{lem}
Recall that a Lie algebra $L$ is said to be capable if $L\cong H/Z(H)$ for a Lie algebra $H$.
The following lemma gives the structure of all  nilpotent Lie algebras and the capability of them  when the derived subalgebra has dimension one.
\begin{lem}\label{5} \cite[Theorem 3.5]{ni3}
if $ L $ be a nilpotent Lie algebra of dimension $ n $ such that $ \dim L^2=1, $ then $ L\cong H(m)\oplus A(n-2m-1) $
for some $ m\geq 1  $ and $ L $ is capable if and only if $m=1,  $ that is, $  L\cong H(1)\oplus A(n-3).$
\end{lem}
The following result gives the Schur multiplier of  nilpotent Lie algebras  with the derived subalgebra of dimension  one.
\begin{prop}\label{d}
Let $ L $ be a nilpotent Lie algebra of dimension $ n $ such that $ \dim L^2=1 $. Then $L\cong H(m)\oplus A(n-2m-1).$
\begin{itemize}
\item[$(i)$] If $ m=1,$ then $ \dim\mathcal{M}(L)=\dfrac{1}{2}(n-1)(n-2)+1=\dfrac{1}{2}n(n-3)+2.$
\item[$(ii)$] If $ m\geq 2,$ then $ \dim\mathcal{M}(L)=\dfrac{1}{2}(n-1)(n-2)-1.$
\end{itemize}

\end{prop}
\begin{proof}
The result follows from \cite[Theorem 3.1]{ni}.
\end{proof}
The next lemma  shows the kernel of  the commutator map  $ \kappa' :L\wedge L \rightarrow L^2$  given by $l\wedge l_1\mapsto [l,l_1]$ is  isomorphic to the Schur multiplier of $L$.
\begin{lem}\cite[Theorem 35 $(iii)$]{elis1}\label{j1}
Let $ L $ be a Lie algebra. Then
$0\rightarrow \mathcal{M}(L)\rightarrow L\wedge L \xrightarrow{\kappa'} L^2\rightarrow 0$ is exact.
\end{lem}
Let $ Z^{\wedge}(L) $ be the exterior center of $L$, which is defined in \cite{ni3}. It is known that $L$ is capable if and only if $Z^{\wedge}(L)=0$. The following lemma is a criterion for detecting the capable Lie algebras.
\begin{lem}\cite[Corollary 2.3]{ni3}\label{j2}
$N \subseteq Z^{\wedge}(L)$ if and only if the natural map $L \wedge L \rightarrow L/N \wedge L/N$ is an isomorphism.\end{lem}
\begin{lem}\label{15555}
 \cite[Theorem 2.7]{ni3}
Let $ A$ and $ B $ be two Lie algebras. Then
\[ (A\oplus B)\wedge (A\oplus B)\cong   (A\wedge A) \oplus  (B\wedge B) \oplus (A^{(ab)}\otimes B^{(ab)}).\]
\end{lem}

\begin{lem}\label{j2222}
Let $L$ be a Lie algebra. Then \[(L\wedge L)^2=\langle [l_1,l_2]\wedge [l_3,l_4]|l_1,l_2,l_3,l_4\in L\rangle= \langle l\wedge l'|l,l'\in L^2\rangle. \]
\end{lem}
\begin{proof}
Since for every $ l_1,l_2,l_3,l_4\in L, $
 $[l_1\wedge l_2, l_3\wedge l_4]=[l_1, l_2]\wedge [l_4,l_3] \in \langle [l_1,l_2]\wedge [l_3,l_4]|l_1,l_2,l_3,l_4\in L\rangle,$ we get $(L\wedge L)^2=\langle [l_1,l_2]\wedge [l_3,l_4]|l_1,l_2,l_3,l_4\in L\rangle= \langle l\wedge l'|l,l'\in L^2\rangle.$
\end{proof}
Another notion having relation to capability is the epicenter $Z^{*}(L)$ of $L$, which is defined in \cite{ala}. From \cite{ni3}, $Z^{*}(L)=Z^{\wedge}(L)$
which is an interesting result. The following result gives a criteria for recognizing capable Lie algebras.
\begin{prop}\label{dd}

Let $ L $ be a finite dimensional Lie algebra with the central ideal $ I$ of $ L.$ Then
\begin{itemize}
\item[$(i)  $]$ \dim (\mathcal{M}(L))\geq \dim( \mathcal{M}(L/I))-\dim (L^{2}\cap I),$
\item[$(ii)  $]$ \dim (\mathcal{M}(L))=\dim( \mathcal{M}(L/I))-\dim (L^{2}\cap I) $ if and only if $ I\subseteq Z^{*}(L).$
\end{itemize}
\end{prop}
\begin{proof}
The result follows from \cite[Proposition 4.1(iii) and Theorem 4.4]{ala}.
\end{proof}
Recall that a Lie algebra is called unicentral provided that $Z(L)=Z^{\wedge} (L).$  For a stem Lie algebra $($a Lie algebra such that $Z(L)\subseteq L^2$)  class 3 and the derived subalgebra of dimension 2, we have
\begin{prop}\label{gk}\cite[Corollary 4.11]{twoo}
Let $ T $ be an $ n$-dimensional stem algebra    of class $3  $ and $ \dim T^2=2.$ Then $ T $ is non-capable if and only if $ n\geq 6.$
Moreover, $T $ is unicentral.
\end{prop}

\begin{lem}\label{ggg}\cite[Lemma 4.5]{twoo}
   Let $ T $ be an $ n $-dimensional stem Lie algebra   of class $ 3 $  such that $ \dim T^2=2.$ Then $ Z(T)=T^3\cong A(1)$ and $T/Z(T)\cong H(1)\oplus A(n-4).$
  \end{lem}
Using notation and terminology in \cite{cic,gon,Gr,had1,hard}, we have
\[L_{4,3}\cong L(3,4,1,4)\cong\langle x_1,\ldots,x_4|[x_1, x_2] = x_3, [x_1, x_3] = x_4\rangle.\]
\[  L_{5,5}\cong L(4,5,1,6)\cong \langle x_1,\ldots,x_5| [x_1, x_2] = x_3, [x_1, x_3] = x_5, [x_2, x_4] = x_5\rangle.\]
\[ L_{5,8}=\langle  x_1,\ldots, x_5\big{|}[x_1, x_2] = x_4, [x_1, x_3] = x_5\rangle.\]
\[ L_{6,22}(\epsilon)=\langle   x_1,\ldots, x_6\big{|}[x_1, x_2] = x_5= [x_3,x_4], [x_1,x_3] = x_6,[x_2, x_4] = \epsilon x_6 \rangle,\]  where $\epsilon \in \mathbb{F} /(\overset{*+}{ \sim})$ and  $\mathrm{Char}~\mathbb{F} \neq 2.$
\[ L_{6,7}^{(2)}(\eta)\cong \langle   x_1,\ldots, x_6\big{|}[x_1, x_2] = x_5, [x_3, x_4]=x_5+x_6, [x_1, x_3] = x_6,[x_2, x_4] = \eta x_6 \rangle,\]  where $\eta \in \{0, \omega\}$ and  $~\mathrm{Char}~\mathbb{F} = 2.$
\[L_1\cong   \langle x_1,\ldots, x_7\big{|}[x_1, x_2] = x_6=[x_3, x_4], [x_1, x_5] = x_7= [x_2, x_3]\rangle.\]
 Then
  \begin{prop}\label{mul}
The Schur multiplier of Lie algebras $L_{5,8},L_{6,22}(\epsilon),  $  $L_{6,7}^{(2)}(\eta),L_1,$  $L_{4,3},L_{5,5}  $ are abelian Lie algebras of dimension $6, 8,8,9,2$ and $4,$ respectively.
\end{prop}
\begin{proof}
The result follows from \cite{had1,hard}, \cite[Proposition 2.10]{ni4} and \cite[Proposition 3.3]{twoo}.
\end{proof}

 Let $ cl(L) $ be the nilpotency class of $ L. $ The following theorem gives the classification of all capable finite dimensional nilpotent Lie algebras with the derived subalgebra of dimension at most two.
\begin{thm}\cite[Theorem 5.5]{twoo}\label{kkkkkk}
Let $ L $ be an $ n$-dimensional  nilpotent Lie algebra such that $\dim L^2\leq 2.$  Then $ L $ is capable is if and only if $ L $ is isomorphic to one the following Lie algebras.
\begin{itemize}
\item[(i)] If $\dim L^2=0,$ then $ L\cong A(n) $ and $ n>1.$
 \item[(ii)] If $\dim L^2=1,$ then $ L\cong H(1)\oplus A(n-3).$
\item[(iii)]If $\dim L^2=2$ and $cl(L)=2,$ then $ L\cong L_{5,8} \oplus A(n-5),$ $ L=L_{6,7}^{(2)}(\eta) \oplus A(n-6),$ $ L\cong L_{6,22}(\epsilon) \oplus A(n-6),$ or $L\cong L_1\oplus A(n-7).$
\item[(iv)]If $\dim L^2=2$ and $cl(L)=3,$ then $L\cong L_{4,3}\oplus  A(n-4) $ or  $L\cong L_{5,5}\oplus  A(n-5).$
\end{itemize}
\end{thm}
\section{main results}
In this section, we are going to compute some functors, such as the Schur multiplier, the exterior square and the tensor square of all nilpotent  Lie algebras with the derived subalgebra of dimension at most $ 2.$ Moreover, we obtain  the corank, $t(L),  $ of all such Lie algebras $L$, where $t(L)=\dfrac{1}{2} n(n-1)-\dim\mathcal{M}(L).$

 Since by Theorem \ref{kkkkkk}, the structure of all capable nilpotent  Lie algebras with the derived subalgebra of dimension at most $ 2$ is well-known,  first of all, we focus on non-abelian non-capable nilpotent  Lie algebras  with the derived subalgebra of dimension at most $ 2.$ Then
by classification of capable nilpotent  Lie algebras with the derived subalgebra of dimension at most $ 2$, we
compute various  functors for these Lie algebras.

 The following lemma shows that the capability of a finite dimensional non-abelian Lie algebra depends only to the capability of a stem subalgebra as a direct summand.
 \begin{prop}  \label{811}
Let $ L $ be a finite dimensional nilpotent Lie algebra of nilpotency class $c>1.$ Then $ L\cong T\oplus A $ and $ Z(T)=Z(L)\cap L^2$ where $ A $ is abelian and $ T $ is stem. Moreover, $ Z^{\wedge}(L)=Z^{\wedge}(T). $
\end{prop}
\begin{proof} By using \cite[Proposition 3.1]{pair}, $Z(T)=Z(L)\cap L^2$ and $ L\cong T\oplus A. $ Thus $L^2=T^2$ and so $Z(T)=Z(L)\cap L^2\subseteq T^2.$ Hence $T$ is a stem. The rest of proof obtained directly by using  \cite[Proposition 3.1]{pair}.
\end{proof}
The nilpotency class of a Lie algebra $L$ with the derived subalgebra of dimension two is $ 2 $ or $ 3.$ Now for a Lie algebra $L$ of class two, by using Proposition \ref{811}, we have $L\cong T\oplus A$ for a stem Lie algebra and an abelian Lie algebra $A.$
  Therefore when $L$ is a nilpotent Lie algebra of class two, $T$ should be a generalized Heisenberg of rank $m$ (a Lie algebra such that $ T^2=Z(T) $ and $ \dim T^2=m$).

The following result gives the Schur multiplier of a non-capable Heisenberg Lie algebra.
 \begin{prop}\cite[Proposition 3.1]{ni4}\label{12}
Let $H$ be a non-capable generalized Heisenberg Lie algebra of rank $ 2 $ such that $ \dim H=n. $ Then
$\dim\mathcal{M}(H)=\dfrac{1}{2} (n-3)(n-2)-2$, or  $\dim\mathcal{M}(H)=\dfrac{1}{2} (n-4)(n-1)+1=\dfrac{1}{2}(n-3)(n-2).$
\end{prop}

Proposition \ref{gk} shows that all  $ n$-dimensional stem Lie algebras $ T $ of class $ 3 $ when $n\geq 6  $ and  $ \dim T^2=2$ are non-capable. Therefore the Schur multiplier of non-capable stem Lie algebra $T$ of class $3$ with $\dim T^2=2$ is obtained as following.
\begin{thm}\label{121}
Let $T$ be an $ n$-dimensional stem Lie algebra of class $ 3 $ such that  $ \dim T^2=2.$ Then for all $n\geq 6,$ we have
\[\mathcal{M}(T)\cong A(\dfrac{1}{2}(n-3)(n-2)).\]
\end{thm}
\begin{proof}
By Proposition \ref{gk} and Lemma \ref{ggg}, $ T $ is non-capable and unicentral and so $ Z^{*} (T)=Z(T)\cong A(1).$
Using Proposition \ref{dd} $ (ii),$ we have $ \dim\mathcal{M}(T)=\dim\mathcal{M}(T/Z^{*} (T))-1.$ By Lemma \ref{ggg}, we have  $T/Z^{*} (T)\cong H(1) \oplus A(n-4).$ Now, Proposition \ref{d} $ (i) $ implies
 \[\dim\mathcal{M}(T)=\dfrac{1}{2} (n-1)(n-4)+1=\dfrac{1}{2}(n-3)(n-2),\]as required.
\end{proof}

In the next proposition, we obtain the corank of a non-capable stem Lie algebra with the derived subalgebra of dimension $ 2. $
\begin{prop}
Let $ T $ be a  non-capable $ n$-dimensional nilpotent stem Lie algebra such that  $ \dim T^2=2.$ Then
\begin{itemize}
\item[$(i)$] If  $cl(T)=2,$ then $t(T)=2n-1 ~~\text{or} ~~t(T)=2n-3.$
\item[$(ii)$]  If $cl(T)=3,$ then $t(T)=2n-3.$
\end{itemize}
\end{prop}
\begin{proof}
 $(i).$ We know that \[t(T)= \dfrac{1}{2} n(n-1)-\dim\mathcal{M}(T).\]
By Proposition \ref{12}, we have \[t(T)= \dfrac{1}{2} n(n-1)-\dfrac{1}{2} (n-3)(n-2)+2\] or \[t(T)= \dfrac{1}{2} n(n-1)-\dfrac{1}{2} (n-4)(n-1)-1.\] Therefore
\[t(T)=2n-1 ~~\text{or} ~~t(T)=2n-3.\]
 $(ii).$ Similar to the case $(i),$ by using Theorem \ref{121}, we can obtain that $ t(T)=2n-1.$
\end{proof}
\begin{thm}\label{h4}
 Let $L$ be a non-capable  $n$-dimensional Lie algebra of class two such that   $ \dim (L/Z(L))=m $ and $ \dim L^2= 2.$ Then   $L\cong H\oplus A(n-m-2) $,  where $ H $ is a non-capable generalized Heisenberg of rank two and $ \mathcal{M}(L)\cong A(\dfrac{1}{2}(n-3)(n-2)-2) $ or  $ \mathcal{M}(L)\cong A(\dfrac{1}{2}(n-3)(n-2)).$
\end{thm}
\begin{proof}
Propositions \ref{811}  implies  $ L\cong H\oplus A(n-m-2), $ where $Z(H)=H^2=L^2\cong A(2). $ Using
 Proposition \ref{12}, we have
\[\dim\mathcal{M}(H)=\dfrac{1}{2} m(m-1)-2\] or \[\dim\mathcal{M}(H)=\dfrac{1}{2} (m+1)(m-2)+1.\]
Now, Lemmas \ref{1} and \ref{ss} imply
 \begin{align*}
\dim \mathcal{M}(L)&=\dim \mathcal{M}(H)+\dim \mathcal{M}(A(n-m-2))+\dim ((H )^{(ab)} \otimes A(n-m-2)).
\end{align*}
Thus \begin{align*}&\dim \mathcal{M}(L)=\dfrac{1}{2} m(m-1)-2+\dfrac{1}{2}(n-m-2)(n-m-3)+m(n-m-2)\\&=\dfrac{1}{2} m(m-1)+\dfrac{1}{2}(n-m-2)(n+m-3)-2=
\dfrac{1}{2}(n-2)(n-3)-2\end{align*}or
\begin{align*}&\dim \mathcal{M}(L)=\dfrac{1}{2} (m+1)(m-2)+1+\dfrac{1}{2}(n-m-2)(n-m-3)+m(n-m-2)\\=&\dfrac{1}{2} (m+1)(m-2)+\dfrac{1}{2}(n-m-2)(n+m-3)+1=\dfrac{1}{2}(n-2)(n-3).\end{align*}
 Hence the result follows.
\end{proof}
\begin{thm}\label{h44}
  Let $L$ be a non-capable  $n$-dimensional Lie algebra of class $ 3 $ such that  $ \dim (L/Z(L))=m $ and $ \dim L^2= 2.$ Then   $L\cong T\oplus A(n-m-1) $ and  $ \mathcal{M}(L)\cong A(\dfrac{1}{2}(n-2)(n-3)). $
\end{thm}
\begin{proof}
Lemma \ref{ggg}, Propositions  \ref{gk} and \ref{811}    imply  $ L\cong T\oplus A(n-m-1), $ where    $Z(T)=L^2\cap Z(L)\cong A(1). $
 Theorem \ref{121} implies
 \[\dim\mathcal{M}(T)=\dfrac{1}{2} m(m-3)+1.\]
Using Lemmas \ref{1} and \ref{ss}, we have
 \begin{align*}
\dim \mathcal{M}(L)&=\dim \mathcal{M}(T)+\dim \mathcal{M}(A(n-m-1))+\dim ((T )^{(ab)} \otimes A(n-m-1)).
\end{align*}
Therefore
\begin{align*}
&\dim \mathcal{M}(L)=\dfrac{1}{2} m(m-3)+1+\dfrac{1}{2}(n-m-1)(n-m-2)+(m-1)(n-m-1)=\\& \dfrac{1}{2} m(m-3)+\dfrac{1}{2}(n-m-1)(n+m-4)+1=\dfrac{1}{2}(n-2)(n-3).\end{align*}
 The result follows.
\end{proof}
\begin{cor}
Let $ L$ be a  non-capable $ n$-dimensional nilpotent  Lie algebra such that  $ \dim L^2=2.$ Then
\begin{itemize}
\item[$(i)$] If  $cl(L)=2,$ then $t(L)=2n-1 ~~\text{or} ~~t(L)=2n-3.$
\item[$(ii)$]  If $cl(L)=3,$ then $t(L)=2n-3.$
\end{itemize}
\end{cor}
\begin{proof}
The result follows by using Theorems \ref{h4} and \ref{h44}.
\end{proof}
\begin{lem}\label{fg}
Let $L$ be a nilpotent Lie algebra of class two. Then $ L\wedge L\cong \mathcal{M}(L)\oplus L^2. $
\end{lem}
\begin{proof}
 Since \[[l_1\wedge l_2, l_3\wedge l_4]=[l_1, l_2]\wedge [l_4,l_3] =l_1\wedge [l_2, [l_4, l_3]]-l_2\wedge [l_1, [l_4, l_3]]=0,\] for all $ l_1,l_2, l_3, l_4 \in L,$ we have $ (L\wedge L)^2=0. $ Therefore $ L\wedge L$ is abelian. By invoking Lemma \ref{j1}, we have $L\wedge L\cong \mathcal{M}(L)\oplus L^2,$ as required.
\end{proof}
\begin{thm}\label{j22}
Let $ L $ be a nilpotent Lie algebra of class two such that $ \dim L/L^2=m $. Then
$ L\otimes L\cong L\wedge L \oplus A(\frac{1}{2}m(m+1)).$
Moreover, if $ cl(L)=2, $ then $ L\otimes L\cong \mathcal{M}(L)\oplus L^2\oplus L\square L \cong   \mathcal{M}(L)\oplus L^2\oplus A(\frac{1}{2}m(m+1)).$
\end{thm}
\begin{proof}
 Using \cite[Lemmas 2.2 and 2.3]{ni5}, we have $L/L^2\square L/L^2\cong A(\frac{1}{2}m(m+1)).$ By using \cite[Theorem 2.5]{ni5}, we conclude $ L\otimes L\cong L\wedge L \oplus A(\frac{1}{2}m(m+1)).$ If $ cl(L)=2, $ then
 \[ L\otimes L\cong \mathcal{M}(L)\oplus L^2\oplus L\square L \cong   \mathcal{M}(L)\oplus L^2\oplus (L/L^2\square L/L^2),\] by using Lemma \ref{fg}.
Hence  the result follows.
\end{proof}

\begin{cor}
Let $L$ be a non-capable $n$-dimensional Lie algebra of class two such that $ \dim L^2= 2.$
Then
$L\otimes L\cong A((n-2)^2)~\text{ or}~ L\otimes L\cong A(n^2-4n+6).$
\end{cor}
\begin{proof}
By using Theorems   \ref{h4} and \ref{j22}, we have
\begin{align*}
&L\otimes L \cong A(\dfrac{1}{2} (n-3)(n-2)-2)\oplus A(2)\oplus  A(\frac{1}{2}(n-2)(n-1))\cong A((n-2)^2),
\end{align*}
or
\begin{align*}
&L\otimes L \cong  A(\dfrac{1}{2} (n-3)(n-2))\oplus A(2)\oplus  A(\frac{1}{2}(n-2)(n-1))\cong A(n^2-4n+6),
\end{align*}
 as required.
\end{proof}
\begin{thm}\label{jjk1}
Let $T$ be an $ n$-dimensional stem Lie algebra of class $ 3 $ such that  $ \dim T^2=2$ and $n\geq 6.$ Then
\begin{itemize}
\item[$(i)  $]$ T\wedge T\cong \mathcal{M}(T)\oplus T^2\cong A(\dfrac{1}{2} (n-2)(n-3)+2).$
\item[$(ii)$]$T\otimes T\cong   A(n^2-4n+6).$
\end{itemize}
 \end{thm}
\begin{proof}
\begin{itemize}
\item[$(i)$]
By Lemma \ref{j2} and Proposition \ref{gk}, we have
$ Z^{\wedge}(T)=Z(T) $ and so $ T\wedge T \cong  T/Z(T)\wedge T /Z(T). $ Since $ cl(T/Z(T))=2, $ $ T\wedge T \cong  T/Z(T)\wedge T /Z(T) $ is abelian, by  using Lemma \ref{fg}. Now, Lemma \ref{j1} implies that $  T\wedge T\cong \mathcal{M}(T)\oplus T^2. $
Using  Theorem \ref{121}, we get
\[  T\wedge T\cong A(\dfrac{1}{2} (n-2)(n-3))\oplus A(2)\cong A(\dfrac{1}{2} (n-2)(n-3)+2).\]
\item[$(ii)$]
Using \cite[Theorem 2.5]{ni5} and  the case $ (i), $ we have $
T\otimes T\cong T\wedge T\oplus T\square T\cong   T \wedge T\oplus A(\dfrac{1}{2} (n-2)(n-1))$ and so $ T\otimes T\cong A(n^2-4n+6) .$

\end{itemize}
\end{proof}

\begin{thm}
Let $L$ be a  non-capable $n$-dimensional Lie algebra of class $ 3 $ such that   $ \dim L^2= 2.$
Then
\begin{itemize}
\item[$  (i)$]$ L\wedge L \cong \mathcal{M}(L)\oplus L^2 \cong A(\dfrac{1}{2} (n-2)(n-3)+2). $
\item[$  (ii)$]$L\otimes L\cong ( L\wedge L)\oplus (L\square L) \cong A(n^2-4n+6) .$\end{itemize}
\end{thm}
\begin{proof}
By Lemma \ref{15555}, Theorems \ref{h44} and \ref{jjk1},  $L\wedge L$ is abelian and so  $ L\wedge L\cong  \mathcal{M}(L)\oplus L^2. $
 Thus
\begin{align*}
&L\wedge L \cong \mathcal{M}(L)\oplus L^2 \cong A(\dfrac{1}{2} (n-2)(n-3)+2).
\end{align*}
 By using case $ (i) $ and  Theorem \ref{j22},  we have
\begin{align*}
&L\otimes L \cong  A(\dfrac{1}{2} (n-2)(n-3)+2))\oplus A(\dfrac{1}{2}(n-2)(n-1))\cong A(n^2-4n+6),
\end{align*}
 as required.
\end{proof}
We determined the Schur multiplier, the tensor square, the exterior square and the corank of all nilpotent Lie algebras with the derived subalgebra of dimension $2.$ Now we are going to find the same for a non-capable Lie algebra with the derived subalgebra of dimension $1.$
\begin{thm}
Let $L$ be a non-capable $n$-dimensional Lie algebra of class two such that    $ \dim L^2= 1.$ Then $t(L)=n,$ $ L \wedge L \cong A(\dfrac{1}{2}(n-1)(n-2))$ and $  L \otimes L \cong A( (n-1)^2).$
\end{thm}
\begin{proof}
The result follows from Lemma  \ref{5}, Proposition \ref{d},  Lemma \ref{fg} and Theorem \ref{j22}.
\end{proof}
The corank, the Schur multiplier, the exterior square and the tensor square of  the class of all non-abelian  nilpotent Lie algebras with the derived subalgebra at most 2 are computed when they are non-capable. Here we are going to obtain the same results for all capable Lie algebras in this class.

We need the following result for the next investigation.
\begin{thm}\label{mull}
We have
\begin{itemize}
\item[(a)]$ L_{5,8}\wedge L_{5,8}\cong A(8) $ and $ L_{5,8}\otimes L_{5,8}\cong A(14). $
\item[(b)]$L_{6,22}(\epsilon)\wedge L_{6,22}(\epsilon)\cong A(10) $ and $ L_{6,22}(\epsilon)\otimes L_{6,22}(\epsilon)\cong A(20). $
\item[(c)]$L_{6,7}^{(2)}(\eta)\wedge L_{6,7}^{(2)}(\eta)\cong A(10) $ and $ L_{6,7}^{(2)}(\eta)\otimes L_{6,7}^{(2)}(\eta)\cong A(20). $
\item[(d)]$ L_1\wedge L_1\cong A(11) $ and $ L_1\otimes L_1\cong A(26). $
\item[(e)]$ L_{4,3}\wedge L_{4,3}\cong A(4) $ and $ L_{4,3}\otimes L_{4,3}\cong A(7). $
\item[(f)]$ L_{5,5}\wedge L_{5,5}\cong A(6) $ and $ L_{5,5}\otimes L_{5,5}\cong A(12). $
\end{itemize}
\end{thm}
\begin{proof}
In the cases $ (a),(b), (c),(d), $ the result follows from
Proposition \ref{mul}, Lemma
\ref{fg} and Theorem \ref{j22}.
In the cases $ (e),(f), $ Lemma \ref{j2222} implies
$ (L_{4,3} \wedge L_{4,3} )^2=0$ and $ (L_{5,5} \wedge L_{5,5} )^2=0.$ Hence
the result follows from using  Proposition \ref{mul} and Theorem
\ref{j22}.
\end{proof}
\begin{thm}\label{kkk}
Let $ L $ be anon-abelian capable $ n$-dimensional  nilpotent Lie algebra  such that $\dim L^2\leq 2.$  Then
\begin{itemize}

 \item[$(i)$] If $\dim L^2=1,$ then $ L\cong H(1)\oplus A(n-3), $ $\mathcal{M}(L)\cong  A(\dfrac{1}{2} (n-1)(n-2)+1), t(L)=n-2$,
 $ L\wedge L\cong A(\dfrac{1}{2} (n-1)(n-2)+2) $ and $ L\otimes L\cong A(n^2-2n+3). $
 \item[$(ii)$]If $\dim L^2=2$ and $cl(L)=2.$ Then one of the following cases holds
  \begin{itemize}
 \item[(a)]$ L\cong  L_{5,8}\oplus A(n-5), $ $\mathcal{M}(L)\cong A(\dfrac{1}{2} n(n-5)+6), t(L)=2n-6$, $ L\wedge L\cong A(\dfrac{1}{2} (n-2)(n-3)+5) $ and $ L\otimes L\cong A(n^2-4n+9). $
\item[(b)]$ L\cong L_{6,22}(\epsilon)\oplus A(n-6),$ $\mathcal{M}(L)\cong A(\dfrac{1}{2} (n+1)(n-6)+8),t(L)=2n-5$, $L \wedge L \cong A(\dfrac{1}{2} (n+1)(n-6)+10) $ and $ L\otimes L\cong A(n^2-4n+8). $
\item[(c)] $ L\cong L_{6,7}^{(2)}(\eta)\oplus A(n-6),$ $\mathcal{M}(L)\cong A(\dfrac{1}{2} (n+1)(n-6)+8),t(L)=2n-5$, $L \wedge L \cong A(\dfrac{1}{2} (n+1)(n-6)+10) $ and $ L\otimes L\cong A(n^2-4n+8). $
\item[(d)]$ L\cong  L_1\oplus A(n-7), $ $\mathcal{M}(L)\cong A(\dfrac{1}{2} (n+2)(n-7)+9),t(L)=2n-2$, $ L\wedge L\cong A(\dfrac{1}{2}(n+2)(n-7)+11) $ and $ L\otimes L\cong A(n^2-4n+5 ). $
 \end{itemize}
 \item[(iii)]If $\dim L^2=2$ and $cl(L)=3.$ Then one of the following cases holds
 \begin{itemize}
 \item[$  (a)$]$ L\cong L_{4,3}\oplus  A(n-4), $ $\mathcal{M}(L)\cong A(\dfrac{1}{2} (n-1)(n-4)+2),t(L)=2n-4,$ $L\wedge L \cong A(\dfrac{1}{2}(n-1)(n-4)+4)$ and $ L\otimes L \cong A(n^2-4n+7).$
  \item[$  (b)$] $L\cong L_{5,5}\oplus  A(n-5), $ $\mathcal{M}(L)\cong A(\dfrac{1}{2} n (n-5)+4),t(L)=2n-4$, $ L\wedge L\cong A(\dfrac{1}{2}n (n-5)+6)$ and
 $ L\otimes L\cong A(n^2-4n+7).$
 \end{itemize}
\end{itemize}
\end{thm}
\begin{proof}

 The case $(i),  $ Theorem \ref{kkkkkk}   $(ii) $ implies  $ L\cong H(1)\oplus A(n-3).$ The result follows from Proposition \ref{d}, Lemma \ref{fg} and Theorem \ref{j22}.\\
In the cases $(ii)$  and $(iii),  $  Theorem \ref{kkkkkk}  $(iii)$  and $(iv)  $ implies $ L\cong L_{5,8} \oplus A(n-5),$ $ L\cong L_{6,7}^{(2)}(\eta) \oplus A(n-6),$ $ L\cong L_{6,22}(\epsilon) \oplus A(n-6),$  $L\cong L_1\oplus A(n-7),$
  $L\cong L_{4,3}\oplus  A(n-4) $ or  $L\cong L_{5,5}\oplus  A(n-5).$ The rest of proof is obtained from
 Lemmas \ref{1}, \ref{15555}, Theorems \ref{j22} and \ref{mull}.
\end{proof}
 The functors, such as the Schur multiplier, the exterior square and the tensor square of all non-abelian nilpotent  Lie algebras with the derived subalgebra of dimension at most $ 2$ are given. When $L$ ia abelian, we have
\begin{thm}
Let $L$ be an abelian Lie algebra of dimension $n$. Then $\mathcal{M}(L)\cong L\wedge L\cong A(1/2n(n-1))$, $t(L)=0$, and
$L\otimes L\cong A(n^2)$.
\end{thm}
\begin{proof} The result is obtained by using  Lemmas \ref{ss} and  \ref{j1}.
\end{proof}


\end{document}